\documentclass[preprint,18pt,nopreprintline]{elsarticle}

\usepackage{amsmath,amssymb,amsfonts,amsthm,mathtools}
\usepackage{graphicx}
\usepackage{hyperref}
\usepackage{enumerate}
\usepackage{tikz}
\usetikzlibrary{positioning,arrows.meta}

\newtheorem{theorem}{Theorem}[section]
\newtheorem{lemma}[theorem]{Lemma}
\newtheorem{corollary}[theorem]{Corollary}
\newtheorem{proposition}[theorem]{Proposition}
\newtheorem{remark}[theorem]{Remark}
\newtheorem{question}[theorem]{Question}
\newtheorem{conjecture}[theorem]{Conjecture}
\newtheorem{definition}[theorem]{Definition}
\newtheorem{example}[theorem]{Example}
\newtheorem{problem}[theorem]{Problem}

\def\bnum{\begin{enumerate}\itemsep=0cm}
\def\enum{\end{enumerate}}
\def\bdf{\begin{definition}\rm }
\def\edf{\end{definition}}
\def\br{\begin{remark}\rm }
\def\er{\end{remark}}
\def\be{\begin{equation}}
\def\ee{\end{equation}}
\def\bt{\begin{theorem}}
\def\et{\end{theorem}}
\def\bl{\begin{lemma}}
\def\el{\end{lemma}}
\def\bc{\begin{corollary}}
\def\ec{\end{corollary}}
\def\bp{\begin{proposition}}
\def\ep{\end{proposition}}
\def\bxa{\begin{example}\rm }
\def\exa{\end{example}}
\def\ba{\begin{array}}
\def\ea{\end{array}}
\def\ben{\begin{eqnarray*}}
\def\een{\end{eqnarray*}}
\def\bdsc{\begin{description}}
\def\edsc{\end{description}}
\def\bpsp{\begin{pspicture}}
\def\epsp{\end{pspicture}}
\def\bea{\begin{eqnarray}}
\def\eea{\end{eqnarray}}
\def\btab{\begin{tabular}}
\def\etab{\end{tabular}}
\def\bpm{\begin{problem}}
\def\epm{\end{problem}}
\def\bfig{\begin{figure}}
\def\efig{\end{figure}}

\def \qe{\hfill \vrule height4pt width 4pt depth 0pt}

\def\1{1\!\hspace{-.08cm}1}

\begin{document}
\begin{frontmatter}

\title{Characterizing uniform hypergraphs via Seidel matrix and Seidel energy}

\author[biu]{Alexander Guterman}
\ead{alexander.guterman@biu.ac.il}

\author[biu]{Shib Sankar Saha\corref{cor1}}
\ead{sahashi@biu.ac.il}

\cortext[cor1]{Corresponding author}
\address[biu]{Department of Mathematics, Bar-Ilan University, Ramat-Gan, 5290002, Israel}

\begin{abstract}
The Seidel energy is defined as the sum of the absolute values of the eigenvalues of the Seidel matrix of a hypergraph. We first characterize the $k$-uniform hypergraphs of fixed order $n$ with minimum and maximum Frobenius norms of their Seidel matrix and then derive bounds for the Seidel energy. We then investigate a hypergraph analogue of \emph{Haemers' Conjecture}, showing that the complete $k$-uniform hypergraph does not minimize Seidel energy. Motivated by the theory of hypoenergetic and non-hypoenergetic graphs, we define Seidel hypoenergetic and Seidel non-hypoenergetic hypergraphs and prove that almost all $k$-uniform hypergraphs are Seidel non-hypoenergetic, where $k\ge3$.
\end{abstract}

\end{frontmatter}

\sloppy
\section{Introduction}
Spectral invariants connect the combinatorial structure of a discrete object with the eigenvalues of associated matrices. In graph theory, besides the adjacency and Laplacian matrices, the Seidel matrix has played an important role in switching theory \cite{WH2012}, equiangular lines \cite{S1996}, and extremal spectral problems \cite{S2020,E2024}. For a graph $G$ of order $n$, the Seidel matrix is $S(G)=J_n-I_n-2A(G)$, where $A(G)$ is the adjacency matrix of $G$, and $J_n$ and $I_n$ are the all-ones and identity matrices of order $n$, respectively. Haemers \cite{WH2012} introduced the Seidel energy of $G$ as the sum of the absolute values of the eigenvalues of $S(G)$ of a graph $G$, denoted by $\mathcal{E}(S(G))$, and also obtained a sharp upper bound for graphs of fixed order. Since then, the Seidel matrix and Seidel energy have been studied from several perspectives; for example, Sz\"oll\H{o}si and \"Osterg{\aa}rd \cite{F2018} considered the enumeration of the Seidel matrix. Berman, Shaked-Monderer, and Singh \cite{A2019} studied complete multipartite graphs determined up to switching by their Seidel spectrum, and Rizzolo \cite{D2019} investigated determinants of the Seidel matrix in connection with a Conjecture of Ghorbani.

For hypergraphs, Feng, Ch'ing, and Li \cite{F1996} introduced an adjacency matrix, and Cardoso, Vecchio, Portugal, and Trevisan \cite{K2022} extended the notion of adjacency energy to hypergraphs. Later, Saha and Panda \cite{SSS2025} studied the effect of hyperedge deletion on hypergraph energy and obtained several bounds in terms of structural parameters. These developments make it natural to investigate a Seidel-type energy for hypergraphs through the adjacency matrix $\mathcal{A}(\mathcal{H})$. More broadly, spectral hypergraph theory has developed along both tensor-based and matrix-based directions. Cooper and Dutle \cite{cooper2012spectra} initiated the tensor approach, while many later works focused on matrix approaches such as adjacency \cite{K2022, SSS2025}, Laplacian \cite{S2022, S2025}, normalized Laplacian \cite{LX2024}, and signless Laplacian matrices \cite{RB2025, Saha2026}. In comparison, the Seidel perspective for hypergraphs is still far less developed. This direction is both natural and challenging, because in a $k$-uniform hypergraph a single hyperedge simultaneously contributes to many vertex pairs through their co-degrees, and hence the Seidel matrix reflects higher-order interactions absent in ordinary graphs. In this paper, we study this hypergraph Seidel matrix from three perspectives: extremal Frobenius norms, extremal Seidel energy, and random $k$-uniform hypergraphs.

The first part of the article concerns the extremal Frobenius norm of their Seidel matrix of hypergraphs. We characterize the $k$-uniform hypergraph that attains the minimum Frobenius norm of the Seidel matrix among all $k$-uniform hypergraphs with the fixed order $n$, and characterize the $k$-uniform hypergraph that attains the maximum Frobenius norm of the Seidel matrix among $k$-uniform hypergraphs with the fixed order $n$. Using the above characterization, we establish lower and upper bounds on their Seidel energy.

The second part of the article studies a hypergraph analogue of \emph{Haemers' Conjecture}. In the graph setting, \emph{Haemers' Conjecture} asserts that among all graphs of fixed order, the complete graph minimizes Seidel energy, and this was later proved by Akbari, Einollahzadeh, Karkhaneei, and Nematollahi \cite{S2020} and, later, with a different approach by Einollahzadeh and Nematollahi \cite{E2024}. It is therefore natural to ask whether the same phenomenon persists for $k$-uniform hypergraphs. We show that this graph-theoretic behavior does not extend to the hypergraph setting: in general, the complete $k$-uniform hypergraph need not minimize Seidel energy among all $k$-uniform hypergraphs of the same order, and thus shows that the hypergraph analogue of \emph{Haemers’ Conjecture} fails. As a part of this analysis, we also study the notion of a Seidel-equitable vertex partition for hypergraphs and the corresponding Seidel quotient matrix, which provides a useful tool for computing Seidel spectra of structured classes of uniform hypergraphs.

The final part of the article is probabilistic in nature and is devoted to the theory of hyperenergetic hypergraphs. In 2007, Nikiforov \cite{VN2007} obtained the asymptotic behavior of the energy of a graph on $n$ vertices in the Erd\H{o}s-R\'enyi random graph model $G(n,p)$. In 2011, using Nikiforov's energy bound, Gutman \cite{IG2011} gave a characterization of non-hypoenergetic graphs in the random graph model $G(n,p)$. Motivated by the classical notions of hypoenergetic and non-hypoenergetic graphs, we introduce the concepts of Seidel hypoenergetic and Seidel non-hypoenergetic hypergraphs, and characterize such $k$-uniform hypergraphs in the random hypergraph model $\mathcal{R}(n,k,p)$, the natural hypergraph analogue of the Erd\H{o}s-R\'enyi model $G(n, p)$.

The article is organized as follows. In Section 2, we introduce the notation and preliminaries used throughout the article. In Section 3, we characterize the extremal Frobenius norms of their Seidel matrix and establish bounds for the Seidel energy of hypergraphs. In Section 4, we study the notion of a Seidel-equitable vertex partition for hypergraphs and use the corresponding quotient matrix to derive Seidel spectral information for structured examples. We then investigate a hypergraph analogue of \emph{Haemers' Conjecture}, showing that the complete $k$-uniform hypergraph does not minimize Seidel energy, in general. In Section 5, we introduce Seidel hypoenergetic and Seidel non-hypoenergetic $k$-uniform hypergraphs and, in the random hypergraph model $\mathcal{R}(n,k,p)$, prove that almost all $k$-uniform hypergraphs are Seidel non-hypoenergetic, where $k\ge3$. In Section 6, we conclude with several open problems.

\section{Preliminaries}
In this section, we collect the notation and basic facts used throughout the article.

For any $k\in\mathbb{Z}^+$ and any set $X$, let $[k]=\{1,2,\ldots,k\}$ and
\[
\binom{X}{k}=\{Y\subseteq X:\ |Y|=k\}.
\]
Also, $\binom{|X|}{k}$ denotes the binomial coefficient, that is, the number of ways to choose $k$ elements from a set $X$ of size $|X|$ without regard to order.

A hypergraph $\mathcal{H}=\big(V(\mathcal{H}),E(\mathcal{H})\big)$ consists of a vertex set $V(\mathcal{H})$ (or simply $V$) and a hyperedge set $E(\mathcal{H})$ (or simply $E$), where $E(\mathcal{H})\subseteq 2^{V(\mathcal{H})}$. For an undirected hypergraph, each hyperedge is an unordered set. A hypergraph $\mathcal{H}$ is said to be {\em finite} if $V(\mathcal{H})$ is a finite set, and the quantity $|V(\mathcal{H})|=n$ is called the {\em order} of $\mathcal{H}$. A hypergraph is $k$-uniform if every hyperedge has size $k$, and a $k$-uniform hypergraph is also called a $k$-graph. In particular, a graph can be regarded as a $2$-uniform hypergraph.

A hypergraph $\mathcal{H}$ is said to be \emph{simple} if it contains no hyperedge of size $1$ and for any distinct $e_i,e_j\in E(\mathcal{H})$, neither $e_i\subset e_j$ nor $e_j\subset e_i$ holds. Let $\mathcal{H}'=\big(V(\mathcal{H}'), E(\mathcal{H}')\big)$ be a hypergraph such that $V(\mathcal{H}')\subseteq V(\mathcal{H})$ and $E(\mathcal{H}') \subseteq E(\mathcal{H})$. Then $\mathcal{H}'$ is called a \emph{subhypergraph} of $\mathcal{H}$.

We use $i\sim j$ to denote that vertices $i$ and $j$ are adjacent. A hypergraph $\mathcal{H}$ is said to be {\em connected} if for any two distinct vertices $i$ and $j$, there exists a sequence of hyperedges $e_1,e_2,\ldots,e_m$ such that $i\in e_1$, $j\in e_m$, and $e_r\cap e_{r+1}\neq\emptyset$ for all $1\le r\le m-1$. For a hypergraph $\mathcal{H}$, let $E_{ij}=\{e\in E(\mathcal{H}): i,j\in e\}$. Then the co-degree of the pair $i,j$ is defined by $c_{ij}=|E_{ij}|$. Throughout the article, all hypergraphs are assumed to be finite and simple unless stated otherwise. Connectedness will be assumed explicitly when needed. For more details on hypergraphs, see \cite{Berge1989, bretto2013, V2009}.

In this article, we study the Seidel matrix of a hypergraph. The natural way to define the Seidel matrix of a hypergraph $\mathcal{H}$ is the following: 
\[
\mathcal{S}(\mathcal{H})=J_n-I_n-2\mathcal{A}(\mathcal{H}),
\]
where $\mathcal{A}(\mathcal{H})$ is the adjacency matrix of $\mathcal{H}$ \cite{F1996}, and $J_n$ and $I_n$ are the all-ones matrix and the identity matrix of order $n$, respectively. Thus, the entries of the Seidel matrix $\mathcal{S}(\mathcal{H})=[\mathcal{S}(\mathcal{H})_{ij}]_{n\times n}$ are given by
\[
\mathcal{S}(\mathcal{H})_{ij}=
\begin{cases}
1-2c_{ij}, & \text{if } i\ne j,\\
0, & \text{if } i=j.
\end{cases}
\]

Let $\rho_1(\mathcal{H}),\rho_2(\mathcal{H}),\dots,\rho_n(\mathcal{H})$ be the eigenvalues of $\mathcal{S}(\mathcal{H})$. These are called the Seidel eigenvalues of $\mathcal{H}$. Since $\mathcal{S}(\mathcal{H})$ is a real symmetric matrix, all of its eigenvalues are real. The notation $\operatorname{Spec}(\mathcal{S}(\mathcal{H}))$
denotes the spectrum of $\mathcal{S}(\mathcal{H})$, the multiset of all Seidel eigenvalues of $\mathcal{H}$.

The \emph{Seidel energy} of a hypergraph $\mathcal{H}$ is denoted by $\mathcal{E}(\mathcal{S}(\mathcal{H}))$ and is defined as the sum of the absolute values of all Seidel eigenvalues of $\mathcal{H}$ \cite{LJK2023, Sahaaa2026},
\[
\mathcal{E}(\mathcal{S}(\mathcal{H}))=\sum_{i=1}^n |\rho_i(\mathcal{H})|.
\]

\section{Extremal Frobenius norms of Seidel matrices of hypergraphs and bounds for Seidel energy}
Let $M = (m_{ij})$ be an $n\times n$ real or complex matrix. The \emph{Frobenius norm} of $M$ is defined by
\[
\|M\|_{F} 
= \left( \sum_{i=1}^{n} \sum_{j=1}^{n} |m_{ij}|^{2} \right)^{\frac{1}{2}}
\]

\subsection{Minimum Frobenius norm of the Seidel Matrix of hypergraphs}
The following lemma establishes a relation between the Frobenius norm of the Seidel matrix and the co-degree of hypergraphs. 

\begin{lemma}\label{lem:Frob}
Let $\mathcal{H}=\big(V(\mathcal{H}), E(\mathcal{H})\big)$ be a $k$-uniform hypergraph with $n$ vertices. Then
\[
\|\mathcal{S}(\mathcal{H})\|_F= \bigg( n(n-1) + 4\!\sum_{i\neq j} c_{ij}^2 - 4\!\sum_{i\neq j} c_{ij}\bigg)^{\frac{1}{2}}.
\]
Moreover,
\[
\sum_{i<j} c_{ij}=|E(\mathcal{H})|\binom{k}{2},
\qquad
\sum_{i\neq j} c_{ij}=2|E(\mathcal{H})|\binom{k}{2},
\]
where $|E(\mathcal{H})|$ is the number of hyperedges in $\mathcal{H}$.
\end{lemma}
\proof By the definition of the Seidel matrix of a hypergraph, we have
\[
\mathcal{S}(\mathcal{H})_{ii}=0,\qquad
\mathcal{S}(\mathcal{H})_{ij}=1-2c_{ij}\quad(i\neq j).
\]
From the Frobenius norm
\[
\|\mathcal{S}(\mathcal H)\|_F=\bigg(\sum_{i,j} \mathcal{S}(\mathcal{H})_{ij}^2\bigg)^{\frac{1}{2}}=\bigg(\sum_{i\neq j}(1-2c_{ij})^2\bigg)^{\frac{1}{2}}.
\]
Hence 
\[
\|\mathcal{S}(\mathcal{H})\|_F=\bigg( n(n-1) + 4\!\sum_{i\neq j} c_{ij}^2 - 4\!\sum_{i\neq j} c_{ij}\bigg)^{\frac{1}{2}}.
\]
Observe that each hyperedge $e\in E(\mathcal H)$ contributes exactly $\binom{k}{2}$ unordered vertex pairs $\{i,j\}\subset e$.

Therefore
\[
\sum_{i<j} c_{ij}
=\sum_{e\in E(\mathcal H)} \binom{k}{2}
=|E(\mathcal H)|\binom{k}{2}.
\]
Since $c_{ij}=c_{ji}$ implies that $\sum\limits_{i\neq j} c_{ij}=2|E(\mathcal H)|\binom{k}{2}$.\\
This completes the proof.\qe

\begin{definition}\rm\cite[Page 3]{bretto2013}\label{D:linear}
A hypergraph $\mathcal{H}=\big(V(\mathcal{H}), E(\mathcal{H})\big)$ is \emph{linear} if it is simple and $|e_i \cap e_j| \leq 1$ for all $e_i, e_j \in E(\mathcal{H})$ with $e_i \neq e_j$, i.e., if no two hyperedges share more than one vertex.
\end{definition}
The following theorem establishes that, among all 
$k$-uniform hypergraphs of fixed order $n$, the Frobenius norm of the Seidel matrix is minimized by the linear $k$-uniform hypergraph.

\begin{theorem}\label{lem:lin}
Let $\mathcal{H}$ be a $k$-uniform hypergraph with $n$ vertices. Then  
$$\|\mathcal{S}(\mathcal{H})\|_F\ge\sqrt{n(n-1)},$$
and equality if and only if $\mathcal{H}$ is linear.
\end{theorem}
\proof For each unordered pair $\{i,j\}$, the co-degree $c_{ij}$ is a nonnegative integer.  Hence, $c_{ij}^2 \ge c_{ij}$ with equality if and only if $c_{ij}\in\{0,1\}$.

Summing over all unordered pairs,
\[
\sum_{i<j} c_{ij}^2 \;\ge\; \sum_{i<j} c_{ij},
\]
and since $c_{ij}=c_{ji}$, finally we have 
\begin{equation}\label{linear1}
\sum_{i\neq j} c_{ij}^2 \;\ge\; \sum_{i\neq j} c_{ij},
\end{equation}
with equality if and only if every $c_{ij}\in\{0,1\}$.

By Lemma \ref{lem:Frob} we have,
\[
\|\mathcal{S}(\mathcal H)\|_F
=\bigg( n(n-1) + 4\!\sum_{i\neq j} c_{ij}^2 - 4\!\sum_{i\neq j} c_{ij}\bigg)^{\frac{1}{2}}
\]
Using equation \eqref{linear1}, $\|\mathcal{S}(\mathcal H)\|_F\ge\sqrt{n(n-1)}$. Equality holds if and only if $\sum\limits_{i\neq j} c_{ij}^2=\sum\limits_{i\neq j} c_{ij}$, i.e., if and only if all $c_{ij}\in\{0,1\}$, which is precisely Definition \eqref{D:linear} of a linear hypergraph.\\
This completes the proof.\qe

\subsection{Maximum Frobenius Norm of the Seidel matrix of hypergraphs}

\begin{definition}\rm\cite[Definition 1.2]{SSS2025}\label{Complete}
A $k$-uniform hypergraph with $n$ vertices is said to be \emph{complete} if it contains all the hyperedges of size $k$. Such a hypergraph is denoted by $\mathcal{C}^k_n$. 
\end{definition}

The following lemma computes the complete Seidel spectrum of $\mathcal{C}^{k}_{n}$.

\begin{lemma}\rm\cite[Theorem 3]{AY2022}\label{HY1}
Let $\mathcal{C}^{k}_{n}$ be the $k$-uniform complete hypergraph with $n>2$ vertices. Then the Seidel spectrum of $\mathcal{C}^{k}_{n}$ is \begin{center}
$\operatorname{Spec}(\mathcal{S}(\mathcal{C}^{k}_{n}))=\small{\Bigg\{\overbrace{2\binom{n-2}{k-2}-1,\ldots,2\binom{n-2}{k-2}-1}^{(n-1)},(n-1)\Big(1-2\binom{n-2}{k-2}\Big)\Bigg\}}$. \end{center}
\end{lemma}
The following theorem establishes that, among all 
$k$-uniform hypergraphs of fixed order $n$, the Frobenius norm of the Seidel matrix is maximized by the $k$-uniform complete hypergraph.

\begin{theorem}\label{lem:frobmax}
Let $\mathcal{H}$ be a $k$-uniform hypergraph with $n>2$ vertices. Then
\[
\|\mathcal{S}(\mathcal{H})\|_F
\leq\Bigg(n(n-1)\left(2\binom{n-2}{k-2}-1\right)^2\Bigg)^{\frac{1}{2}},
\]
with equality if and only if $\mathcal{H} = \mathcal{C}_n^k$.
\end{theorem}

\proof Let $\rho_1(\mathcal{H}),\rho_2(\mathcal{H}),\dots,\rho_n(\mathcal{H})$ be eigenvalues of $\mathcal{S}(\mathcal{H})$. There are two cases.\\\\ 
\textbf{Case (i):} Let $\binom{n-2}{k-2}=1$. Then either $k-2=0$ or $n-2=k-2$.

For $k-2=0\implies k=2$, it is simply a graph case. Therefore, $\mathcal{H}=G$ and $\mathcal{S}(\mathcal{H})=S(G)$. Since every off-diagonal entry of S(G) is $\pm1$, we have
\[
\|S(G)\|_F
= \Big(\sum_{i=1}^n \sum_{j=1}^n (S(G))_{ij}^2\Big)^{\frac{1}{2}}
=\sqrt{n(n-1)},
\]
Thus, $\|S(G)\|_F$ is constant for all graphs of order $n$, which is independent of the structure of $G$. Hence, the equality for the Frobenius norm of the Seidel matrix is attained when the graph is complete ($\mathcal{C}^2_n=K_n$).

If $n=k$, then the only possible $k$-subset of $V(\mathcal{H})$ is $V(\mathcal{H})$ itself. Hence, the unique complete $k$-uniform hypergraph on $n=k$ vertices is $\mathcal{C}_k^k$ consisting of the single hyperedge $V(\mathcal{H})$. By Lemma \ref{HY1}, $\rho_i(\mathcal{C}^k_n)=1$ for all $i=1,\ldots,n-1$, $\rho_n(\mathcal{C}^k_n)=-(n-1)$.     

Hence, 
$$\|\mathcal{S}(\mathcal{C}^k_k)\|_F=\Big(\sum\limits_{i=1}^n \rho_i(\mathcal{C}^k_k)^2\Big)^{\frac{1}{2}}=\sqrt{n(n-1)}$$
\textbf{Case (ii):} Let $\binom{n-2}{k-2}>1$.

For each pair $i\ne j$, the co-degree $c_{ij}$ counts the number of $k$-hyperedges containing $\{i,j\}$. After fixing $i$ and $j$, the remaining $k-2$ vertices of such a hyperedge can be chosen from the other $n-2$ vertices in at most $\binom{n-2}{k-2}$ ways. Therefore,
$$0 \leq c_{ij} \leq \binom{n-2}{k-2},$$
implies, 
$$-1\leq-(1-2c_{ij})\leq-\Bigg(1-2\binom{n-2}{k-2}\Bigg),$$
and hence
\[
|1 - 2c_{ij}|\leq 2\binom{n-2}{k-2}-1.
\]
Since $\mathcal{S}(\mathcal{H})$ is symmetric with zero diagonal and $\|\mathcal{S}(\mathcal{H})\|_F=\Big(2\sum\limits_{i<j}(1-2c_{ij})^2\Big)^{\frac{1}{2}}$ we have

\begin{equation}\label{full}
\|\mathcal{S}(\mathcal{H})\|_F\leq
\bigg(n(n-1)\left(2\binom{n-2}{k-2}-1\right)^2\bigg)^{\frac{1}{2}}.
\end{equation}
Suppose that the equality in the inequality \eqref{full} holds. Then we have $|1-2c_{ij}| = 2\binom{n-2}{k-2}-1$ for all $i \neq j$, giving either $c_{ij} = \binom{n-2}{k-2}$ or $c_{ij} = 0$. Since $\binom{n-2}{k-2}>1$, the case $c_{ij} = 0$ implies $|1 - 0| = 1 < 2\binom{n-2}{k-2}-1$, which is a contradiction to our assumption. Hence equality gives 
$c_{ij}=\binom{n-2}{k-2}$ for all $i \neq j$, i.e., $\mathcal{H}=\mathcal{C}_n^k$.

Conversely let $\mathcal{H}=\mathcal{C}_n^k$. Since $\mathcal{S}(\mathcal{C}_n^k)$ is a real symmetric matrix with eigenvalues $\rho_i(\mathcal{C}^k_n)=2\binom{n-2}{k-2}-1$ for all $i=1,\ldots,n-1$, $\rho_n(\mathcal{C}^k_n)=(n-1)\Big(1-2\binom{n-2}{k-2}\Big)$, by Lemma \ref{HY1}. 

Then
\begin{align*}
\|\mathcal{S}(\mathcal{C}^k_n)\|_F
&=\left(\sum_{i=1}^n \rho_i(\mathcal{C}^k_n)^2\right)^{\frac{1}{2}},\\
&=\Bigg((n-1)\Big(2\binom{n-2}{k-2}-1\Big)^2+(n-1)^2\Big(2\binom{n-2}{k-2}-1\Big)^2\Bigg)^{\frac{1}{2}},\\
&=\Bigg(n(n-1)\left(2\binom{n-2}{k-2}-1\right)^2\Bigg)^{\frac{1}{2}}.
\end{align*}
This completes the proof.\qe

\subsection{Lower and upper bounds for the Seidel energy of hypergraphs}
For $1 \le p < \infty$, the $\ell_{p}$-norm of a vector $x=(x_1,\dots,x_n)\in\mathbb{K}^n$ with $\mathbb{K}\in\{\mathbb{R},\mathbb{C}\}$ is defined by
\[
\|x\|_{p}=\left(\sum_{i=1}^{n}|x_i|^{p}\right)^{\frac{1}{p}}.
\]

\begin{lemma}\rm\cite[Page 333]{H2013}\label{200}
Let $1\le p_1<p_2<\infty$ and $x\in\mathbb{K}^n$ with $\mathbb{K}\in\{\mathbb{R},\mathbb{C}\}$. Then
\[
\|x\|_{p_2}\ \le\ \|x\|_{p_1}\ \le\ n^{\frac1{p_1}-\frac1{p_2}}\ \|x\|_{p_2},
\]
where $\|\cdot\|_{p}$ is the $\ell_{p}$-norm.
\end{lemma}
The following lemma gives a useful comparison between the Seidel energy of a hypergraph and the Frobenius norm of the Seidel matrix.

\begin{lemma}\label{lem:sand}
Let $\mathcal{H}$ be a hypergraph on $n$ vertices. Then
\[
\|\mathcal{S}(\mathcal{H})\|_F
\le
\mathcal{E}(\mathcal{S}(\mathcal{H}))
\le
\sqrt{n}\,\|\mathcal{S}(\mathcal{H})\|_F.
\]
\end{lemma}
\proof  Since $\mathcal{S}(\mathcal{H})$ is symmetric, it is orthogonally diagonalizable. Then there exists an orthogonal matrix $Q$ with $\mathcal{S}(\mathcal{H})=Q~\mathrm{diag}(\rho_1,\dots,\rho_n)~Q^\top$ and $\rho_i\in\mathbb{R}$. For a symmetric matrix, the singular values are $|\rho_1|, \dots, |\rho_n|$. Hence,
\[
\|\mathcal{S}(\mathcal{H})\|_F=\sqrt{\Big(\sum_{i=1}^n \rho_i^2\Big)}.
\]
Consider the vector $x=\big(|\rho_1|,\dots,|\rho_n|\big)\in\mathbb{R}^n_{\ge 0}$. Let $p=1,~p_2=2$ and apply the norm inequalities in Lemma \ref{200},
\[
\|x\|_2 \;\le\; \|x\|_1 \;\le\; \sqrt{n}\,\|x\|_2.
\]
Substituting $\|x\|_2=\sqrt{\sum\limits_{i=1}^n\rho_i^2}=\|\mathcal{S}(\mathcal{H})\|_F$ and $\|x\|_1=\sum\limits_{i=1}^n |\rho_i|=\mathcal{E}(\mathcal{S(H)})$
yields
\[
\|\mathcal{S}(\mathcal{H})\|_F \;\le\;\mathcal{E}(\mathcal{S(H)})  \;\le\; \sqrt{n}\,\|\mathcal{S}(\mathcal{H})\|_F,
\]
This completes the proof.\qe 

The following theorem gives an upper bound for the Seidel energy of a hypergraph.

\begin{theorem}\label{thm:seidel_upper}
Let $\mathcal{H}$ be a $k$-uniform hypergraph with $n>2$ vertices. Then 
\[
\mathcal{E}(\mathcal{S}(\mathcal{H})) 
\leq 
n\sqrt{n-1}\left(2\binom{n-2}{k-2}-1\right),
\]
\end{theorem}

\proof By Theorem \ref{lem:frobmax}, 
\[
\|\mathcal{S}(\mathcal{H})\|_F 
\leq 
\sqrt{n(n-1)}\left(2\binom{n-2}{k-2}-1\right).
\]
Applying the upper bound in Lemma \ref{lem:sand},
\begin{align*}
\mathcal{E}(\mathcal{S}(\mathcal{H})) 
&\leq 
\sqrt{n}\,\|S(\mathcal{H})\|_F \\
&\leq 
\sqrt{n} \cdot \sqrt{n(n-1)}\left(2\binom{n-2}{k-2}-1\right)\\
&=n\sqrt{n-1}\left(2\binom{n-2}{k-2}-1\right).
\end{align*}
This completes the proof.\qe 

The following theorem gives a lower bound for the Seidel energy of a hypergraph.

\begin{theorem}\label{thm:min-proper}
Let $\mathcal{H}$ be a $k$-uniform hypergraph with $n$ vertices. Then the following holds.

\begin{enumerate}
\item[(i)] For every $\mathcal H$,
\[
\mathcal E(\mathcal S(\mathcal H))\ge \sqrt{n(n-1)}.
\]

\item[(ii)] If for some $\mathcal H$ equality holds in (i), then $\mathcal H$ is linear.

\item[(iii)] For every linear $k$-uniform hypergraph $\mathcal H$ on $n$ vertices,
\[
\sqrt{n(n-1)}\le \mathcal E(\mathcal S(\mathcal H))\le n\sqrt{n-1}.
\]
\end{enumerate}
\end{theorem}

\proof By Theorem \ref{lem:lin},
\[
\|\mathcal S(\mathcal H)\|_F\ge \sqrt{n(n-1)},
\]
with equality if and only if $\mathcal H$ is linear.

By Lemma \ref{lem:sand},
\[
\|\mathcal S(\mathcal H)\|_F
\le
\mathcal E(\mathcal S(\mathcal H))
\le
\sqrt n\,\|\mathcal S(\mathcal H)\|_F.
\]
From the left inequality of the above,
\[
\mathcal E(\mathcal S(\mathcal H))\ge \sqrt{n(n-1)}
\]
for every $\mathcal H$, proving (i).

If equality holds in (i), then
\[
\sqrt{n(n-1)}=\mathcal E(\mathcal S(\mathcal H))
\ge
\|\mathcal S(\mathcal H)\|_F
\ge
\sqrt{n(n-1)}.
\]
Hence
\[
\|\mathcal S(\mathcal H)\|_F=\sqrt{n(n-1)}.
\]
By Theorem \ref{lem:lin}, $\mathcal H$ is linear. This proves (ii).

Now let $\mathcal H$ be a connected linear $k$-uniform hypergraph on $n$ vertices. Then Theorem \ref{lem:lin},
\[
\|\mathcal S(\mathcal H)\|_F=\sqrt{n(n-1)}.
\]
Applying Lemma \ref{lem:sand}, we obtain
\[
\sqrt{n(n-1)}
=
\|\mathcal S(\mathcal H)\|_F
\le
\mathcal E(\mathcal S(\mathcal H))
\le
\sqrt n\,\|\mathcal S(\mathcal H)\|_F
=
n\sqrt{n-1}.
\]
This proves (iii).\\
This completes the proof of all the parts of the theorem.\qe 

\begin{remark}
The upper bound in part (iii) of Theorem \ref{thm:min-proper} improves the bound in Theorem \ref{thm:seidel_upper} for linear $k$-uniform hypergraphs. 
\end{remark}

\section{Haemers' Conjecture fails for the Seidel matrix of hypergraphs}

In 2012, Haemers introduced the concept of the Seidel energy of a graph and proposed the following Conjecture: 

\begin{conjecture}\rm\cite[Page 659]{WH2012}\label{CCcc}
The complete graph $K_n$ minimizes the Seidel energy among all graphs with fixed order $n$, i.e., for every graph $G$ of order $n$, $\mathcal{E}(S(G))\geq\mathcal{E}(S(K_n))=2n-2$.
\end{conjecture}
The above Conjecture \ref{CCcc} was open for some time and was investigated in the following articles:

\begin{enumerate}
\item In 2012, Haemers \cite{WH2012} first proved this Conjecture for $n\leq10$.

\item In 2016, Greaves, Koolen, Munemasa, and Szöllősi \cite{GG2016} proved for $ n\leq 12$.

\item In 2016, Oboudi \cite{MRO2016} proved {\em Haemers’ Conjecture} for every $r$-regular graph $G$ of order $n$ such that $r\neq\frac{n-1}{2}$ and $G$ has no eigenvalue in $(-1,0)$.

\item In 2017, Ghorbani \cite{EG2017} proved {\em Haemers’ Conjecture} for the graphs $G$ of order $n$ such that $n-1\leq|det(S(G))|$, where $det(S(G))$ is the determinant of the Seidel matrix $S(G)$.

\item In 2020,  Akbari, Einollahzadeh, Karkhaneei, and Nematollahi \cite{S2020} gave the proof of a {\em Haemers’ Conjecture} on the Seidel energy of graphs. Later, in 2024, Einollahzadeh and Nematollahi \cite{E2024} gave a short proof of {\em Haemers' Conjecture} on the Seidel energy of graphs.
\end{enumerate}
Now it is natural to ask the following question:

\begin{question}\label{Queee}
Is Haemers’ Conjecture valid for all $k$-uniform hypergraphs? Specifically, does the inequality 
\begin{center}
$\mathcal{E}(\mathcal{S}(\mathcal{H}))\geq\mathcal{E}(\mathcal{S}(\mathcal{C}^k_{n}))$
\end{center}
hold for every $k$-uniform hypergraph $\mathcal{H}$ with $n$ vertices? 
\end{question}
\textbf{Answer:} The answer is negative.

To prove this, we first compute the Seidel energy of the complete $k$-uniform hypergraph and then compare it with that of a suitably structured k-uniform hypergraph.

The following theorem computes the Seidel energy of $k$-uniform complete hypergraphs.
\begin{theorem}\label{HY2}
Let $\mathcal{C}^{k}_{n}$ be the $k$-uniform complete hypergraph with $n$ vertices. Then the Seidel energy of $\mathcal{C}^{k}_{n}$ is \begin{center}
$\mathcal{E}(\mathcal{S}(\mathcal{C}^k_n))=2(n-1)\Big(2\binom{n-2}{k-2}-1\Big)$.    
\end{center}
\end{theorem}

\proof By Lemma \ref{HY1},
\[
\operatorname{Spec}(S(C_n^k))=
\Bigg\{
\underbrace{2\binom{n-2}{k-2}-1,\dots,2\binom{n-2}{k-2}-1}_{n-1\text{ times}},
\ (n-1)\left(1-2\binom{n-2}{k-2}\right)
\Bigg\}.
\]
Let $a=2\binom{n-2}{k-2}-1$. Then $a\ge 1$, and the spectrum consists of $a$ with multiplicity $n-1$ and $-(n-1)a$ once. Therefore
\[
\mathcal{E}(S(C_n^k))=(n-1)|a|+|-(n-1)a|
=2(n-1)a
=2(n-1)\left(2\binom{n-2}{k-2}-1\right).
\]
This completes the proof.\qe

\begin{definition}
Let $X=\{x_1,\dots,x_n\}$ be a finite set and let $U\subseteq X$. The characteristic vector of \(U\) with respect to \(X\) is the vector \(\chi_U\in\mathbb{R}^n\) defined by
$$(\chi_U)_i=
\begin{cases}
1,& x_i\in U,\\
0,& x_i\notin U.
\end{cases}$$
\end{definition}

\begin{definition}\label{d1}\rm\cite[Definition 2.3]{B2012}
Let $M$ be a real symmetric matrix whose rows and columns are indexed by $X=\{1, 2,\ldots,n\}$. Let $\pi=\{X_{1},X_{2},\ldots,X_{t}\}$ be a partition of $X$. The characteristic matrix $C$ is the $n\times m$ matrix whose $jth$ column is the characteristic vector of $X_j$ where $j=1, 2,\ldots,m$. Let $M$ be partitioned according to $\pi$ as
$$ M_{\pi} = \begin{bmatrix} 
M_{11} & M_{12} & \dots & M_{1m} \\
M_{21} & M_{22} & \ldots & M_{2m}   \\
\vdots & \vdots & \ddots & \vdots \\
M_{m1} & M_{m2} & \ldots & M_{mm} \\
\end{bmatrix},$$
where $M_{ij}$ denotes the submatrix (block) of $M$ formed by rows in $X_i$ and the columns in $X_j$. If $q_{ij}$ denotes the row sum of $M_{ij}$, then the matrix $Q^M=[(q^M_{ij})]_{m \times m}$ is called the quotient matrix of $M$. If the row sum of each block $M_{ij}$ is a constant, then the partition $\pi$ is called an equitable partition.
\end{definition}

\begin{lemma}\rm\cite[Lemma 2.3.1]{B2012}\label{u1}
Let $Q^M$ be the quotient matrix of any square matrix $M$ corresponding to an equitable partition. Then the spectrum of $M$ contains the spectrum of $Q^M$.
\end{lemma}
In spectral hypergraph theory, eigenvalues associated with various matrix representations of hypergraphs, such as adjacency, Laplacian, signless Laplacian, and Seidel-type matrices, play a fundamental role in encoding structural, combinatorial, and extremal properties of hypergraphs \cite{RB2025, LJK2023, S2022, S2025, LX2024}. However, in contrast to the graph case, there is still no general or unified method to determine the eigenvalues of hypergraphs. This difficulty primarily arises from the multi-vertex nature of hyperedges, which prevents the direct extension of linear and pairwise techniques from classical spectral graph theory.

A significant step toward a systematic spectral analysis of hypergraphs was made in 2022 \cite{S2022}, and later in 2025 \cite{S2025}, under the assumption that the hypergraph admits an equitable vertex partition with respect to its Laplacian matrix. Using Definition \ref{d1}, the notion of an equitable vertex partition for hypergraphs with respect to the Laplacian matrix was introduced in (see \cite[Theorem 4.5]{S2022}). It was demonstrated that such partitions give Laplacian eigenvalues through a suitably defined Laplacian quotient matrix \cite{S2020, S2025}. This framework provided one of the first effective mechanisms for extracting exact Laplacian spectral information.

Despite of this progress, the existing theory is largely restricted to Laplacian-based matrices and does not naturally extend to other important spectral constructions. In particular, the Seidel matrix, which has been highly effective in graph theory for studying switching equivalence, extremal spectra, rigidity, and energy-related parameters, lacks an analogous equitable vertex partition theory in the hypergraph setting.

In the following, we extend the concept of equitable vertex partition to hypergraphs using the Seidel matrix. We introduce a notion of Seidel-equitable vertex partition for hypergraphs and show how it leads to a reduced Seidel-quotient matrix whose eigenvalues are contained in the spectrum of the hypergraph Seidel matrix. This extension provides a new structural tool in spectral hypergraph theory and enables the systematic study of Seidel eigenvalues, Seidel energy, and related extremal problems for hypergraphs.

\begin{definition}
Let $\mathcal{H}=(V,E)$ be a hypergraph and $V_{1},V_{2},\ldots,V_{t}$ be a partition of the vertex set $V$. Then the set $\{V_{1},V_{2},\ldots,V_{t}\}$ is said to form a {\em Seidel-equitable vertex partition} if for each $r,s\in \{1,2,\ldots, t\}$ and for any $i\in V_{r}$, 
\begin{center}
$q^{\mathcal{S}}_{rs}=\sum\limits_{j;j\in V{s}}\mathcal{S}(\mathcal{H})_{ij}$,
\end{center} where $q_{rs}$ are constants depending on $r$ and $s$. 
\end{definition}

\begin{definition}
Let $\{V_1,\dots,V_t\}$ be a Seidel-equitable vertex partition of $V(\mathcal{H})$. The matrix $Q_{\mathcal{H}}^{\mathcal{S}}=[q_{rs}^{\mathcal{S}}]_{t\times t}$ is called the Seidel quotient matrix of $\mathcal{H}$.
\end{definition}
The following lemma is a consequence of Lemma \ref{u1}.

\begin{lemma}\label{P1} 
Let $\mathcal{H}=(V,~E)$ be a hypergraph and $Q^{\mathcal{S}}(\mathcal{H})$ be the Seidel quotient matrix corresponding to a Seidel equitable vertex partition of $V$. Then the Seidel spectrum of $\mathcal{H}$ contains the spectrum of $Q^{\mathcal{S}}(\mathcal{H})$.
\end{lemma}

\begin{definition}\rm\cite[Page 27]{D2023}
Let $\mathcal{H}=\big(V(\mathcal{H}), E(\mathcal{H})\big)$ be a $k$-uniform hypergraph, where $V(\mathcal{H})=\{v_1,v_2,v_3,\ldots,v_{3k-4},v_{3k-3}\}$, $E(\mathcal{H})=\{e_1,e_2,e_3\}$ and $e_1=\{v_{1},v_{2},\ldots,v_{k}\}$, $e_2=\{v_{k},v_{(k+1)},\ldots,v_{(2k-1)}\}$ and $e_3=\{v_{(2k-1)},v_{2k},\ldots,v_{(3k-3)},v_1\}$. Then $\mathcal{H}$ is called the $k$-uniform hypertriangle $\mathcal{T}^k_3$.
\end{definition}
The following lemma computes the Seidel energy of a $k$-uniform hypertriangle.

\begin{lemma}\label{LHY3}
Let $\mathcal{T}^{k}_{3}$ be the $k$-uniform hypertriangle. Then the Seidel spectrum of $\mathcal{T}^{k}_{3}$ is 
\[
\operatorname{Spec}(\mathcal{S}(\mathcal{T}^k_3))=\bigg\{\overbrace{1,1,\ldots,1,1}^{(3k-9)},~3-k\pm\sqrt{k^2-4},~3-k\pm\sqrt{k^2-4},~\frac{k-3\pm\sqrt{k^2+6k-7}}{2}\bigg\}. 
\]
\end{lemma}
\proof Let $V=\{v_1,v_2,v_3,\ldots,v_{(3k-4)},v_{(3k-3)}\}$ be the vertex set of $\mathcal{T}^{k}_{3}$, and $E=\{e_1,e_2,e_3\}$ be  the hyperedge set, where $e_1=\{v_{1},v_{2},\ldots,v_{k}\}$, $e_2=\{v_{k},v_{(k+1)},\ldots,v_{(2k-1)}\}$ and $e_3=\{v_{(2k-1)},v_{2k},\ldots,v_{(3k-3)},v_1\}$.

For each $e_1,e_2,e_3\in E$, we can construct the family of $3k-9$ vectors, where $Z^i_1=[Z^i_1(t)]_{(3k-3)\times1}$, for $3\leq i\leq(k-1)$; $Z^i_2=[Z^i_2(t)]_{(3k-3)\times1}$, for $(k+2)\leq i\leq(2k-2)$; $Z^i_3=[Z^i_3(t)]_{(3k-3)\times1}$, for $(2k+1)\leq i\leq(3k-3)$ such that: 
\begin{center}
$Z^i_1(t) = \left\{ \begin{array}{rcl}
-1  & \mbox{for} & t=v_2 \\ 
1  & \mbox{for} &  t=v_i \\
0 & \mbox{for} & t\in V\setminus\{v_2,v_i\}
\end{array}\right.$
\end{center}

\begin{center}
$Z^i_2(t) = \left\{ \begin{array}{rcl}
-1  & \mbox{for} & t=v_{k+2} \\ 
1  & \mbox{for} &  t=v_i \\
0 & \mbox{for} & t\in V\setminus\{v_{k+2},v_i\}
\end{array}\right.$
\end{center}

\begin{center}
$Z^i_3(t) = \left\{ \begin{array}{rcl}
-1  & \mbox{for} & t=v_{2k+1} \\ 
1  & \mbox{for} &  t=v_i \\
0 & \mbox{for} & t\in V\setminus\{v_{2k+1},v_i\}
\end{array}\right.$
\end{center}
Then we have the following relations:

\begin{center}
$\hspace{-1cm}\mathcal{S}(\mathcal{T}^k_3)~Z^i_1=1~Z^i_1$ for all $i=3,4,\ldots,k-1$;
\end{center}

\begin{center}
$\hspace{.5cm}\mathcal{S}(\mathcal{T}^k_3)~Z^i_2=1~Z^i_2$ for all $i=k+2,k+3,\ldots,2k-2$;
\end{center}

\begin{center}
$\hspace{.8cm}\mathcal{S}(\mathcal{T}^k_3)~Z^i_3=1~Z^i_3$ for all $i=2k+1,2k+2,\ldots,3k-3$.
\end{center}
Each vector in the family has support on exactly two coordinates, with one distinguished negative coordinate and one positive coordinate. Within each family $\{Z^i_1:i=3,4,\ldots,k-1\}$, $\{Z^i_2:i=k+2,k+3,\ldots,2k-2\}$, and $\{Z^i_3:i=2k+1,2k+2,\ldots,3k-3\}$, the positive coordinate uniquely determines the vector, so each family is linearly independent. Since the three families have disjoint distinguished coordinates $v_2,v_{k+2},v_{2k+1}$, their union is also linearly independent. Hence, the total number of linearly independent eigenvectors corresponding to the eigenvalue 
1 is at least $(k-3)+(k-3)+(k-3)=3k-9$.

Let us consider the sets $\mathcal{X}_1=\{v_1\}$, $\mathcal{X}_2=\{v_k\}$, $\mathcal{X}_3=\{v_{2k-1}\}$, $\mathcal{X}_4=\{v_2,\ldots,v_{k-1}\}$, $\mathcal{X}_5=\{v_{k+1},\ldots,v_{2k-2}\}$ and $\mathcal{X}_6=\{v_{2k},\ldots,v_{3k-3}\}$.

Then the set $\{\mathcal{X}_1,\mathcal{X}_2,\mathcal{X}_3,\mathcal{X}_4,\mathcal{X}_5,\mathcal{X}_6\}$ forms a Seidel-equitable vertex partition of the vertex set $V$ of $\mathcal{S}(\mathcal{T}^k_3)$. Let $Q^{\mathcal{S}}(\mathcal{T}^k_{3})=[q^{\mathcal{S}}_{rs}]_{6\times6}$ be the quotient matrix corresponding to $\mathcal{S}(\mathcal{T}^k_3)$.

Then the entries of $q^\mathcal{S}_{rs}$ for $1\leq r,s\leq6$ are determined as follows:
\begin{align*}
&q^{\mathcal{S}}_{11}=q^{\mathcal{S}}_{22}=q^{\mathcal{S}}_{33}=0;\\
\vspace{.5cm}
&q^{\mathcal{S}}_{12}=q^{\mathcal{S}}_{21}=q^{\mathcal{S}}_{13}=q^{\mathcal{S}}_{31}=q^{\mathcal{S}}_{23}=q^{\mathcal{S}}_{32}=-1;\\
\vspace{.5cm}
&q^{\mathcal{S}}_{44}=q^{\mathcal{S}}_{55}=q^{\mathcal{S}}_{66}=-(k-3);\\
\vspace{.5cm}
&q^{\mathcal{S}}_{14}=q^{\mathcal{S}}_{16}=q^{\mathcal{S}}_{24}=q^{\mathcal{S}}_{25}=q^{\mathcal{S}}_{35}=q^{\mathcal{S}}_{36}=-(k-2); \\
\vspace{.5cm}
&q^{\mathcal{S}}_{15}=q^{\mathcal{S}}_{26}=q^{\mathcal{S}}_{34}=q^{\mathcal{S}}_{45}=q^{\mathcal{S}}_{46}=q^{\mathcal{S}}_{54}=q^{\mathcal{S}}_{56}=q^{\mathcal{S}}_{64}=q^{\mathcal{S}}_{65}=(k-2);
\end{align*}

Therefore the matrix $Q^{\mathcal{S}}(\mathcal{T}^k_3)$ is

$$ Q^{\mathcal{S}}(\mathcal{T}^k_3) = \begin{bmatrix} 
0 & -1 & -1 & -(k-2) & (k-2) & -(k-2) \\
    
-1 & 0 & -1 & -(k-2) & -(k-2) & (k-2)  \\
    
-1 & -1 & 0 & (k-2) & -(k-2) & -(k-2) \\
    
-1 & -1 & 1 & -(k-3) & (k-2) & (k-2) \\

1 & -1 & -1 & (k-2) & -(k-3) & (k-2) \\

-1 & 1 & -1 & (k-2) & (k-2) & -(k-3) \\  
\end{bmatrix}$$

Using MATLAB, we obtain that the characteristic polynomial of $Q^{\mathcal{S}}(\mathcal{T}^k_3)$ is 
\begin{center}
$P_{(\mathcal{T}^k_3)}(x)=\Big(x^2+(2k-6)x+(13-6k)\Big)^2\Big(x^2-(k-3)x-(3k-4)\Big)$.    
\end{center}

The roots of the equation $P_{(\mathcal{T}^k_3)}(x)=0$ are $3-k\pm\sqrt{k^2-4}$, $3-k\pm\sqrt{k^2-4}$ and $\frac{k-3\pm\sqrt{k^2+6k-7}}{2}$.

Using Lemma \ref{P1}, we deduce that the spectrum of $Q^{\mathcal{S}}(\mathcal{T}^k_3)$ is included in the Seidel spectrum of $\mathcal{T}^{k}_{3}$. Therefore, the solutions of the equation $P_{(\mathcal{T}^k_3)}(x)=0$ are precisely the Seidel eigenvalues of $\mathcal{T}^{k}_{3}$.

Hence, the Seidel spectrum of $\mathcal{T}^k_3$ is 

\[
\operatorname{Spec}(\mathcal{S}(\mathcal{T}^k_3))=\bigg\{\overbrace{1,1,\ldots,1,1}^{(3k-9)},~3-k\pm\sqrt{k^2-4},~3-k\pm\sqrt{k^2-4},~\frac{k-3\pm\sqrt{k^2+6k-7}}{2}\bigg\}. 
\]\qe

The following theorem gives the explicit formula for the Seidel energy of the $k$-uniform hypertriangle. 

\begin{theorem}\label{HY3}
Let $\mathcal{T}^{k}_{3}$ be the $k$-uniform hypertriangle. Then the Seidel energy of $\mathcal{T}^{k}_{3}$ is \begin{center}
$\mathcal{E}(\mathcal{S}(\mathcal{T}^k_3))=(3k-9)+4\sqrt{k^2-4}+\sqrt{k^2+6k-7}$.    
\end{center}
\end{theorem}
\proof The proof follows from Lemma \ref{LHY3}.\qe

The following theorem answers \textbf{Question \ref{Queee}}.

\begin{theorem}\label{thm:haemers-fails}
For every integer $k\ge 3$, there exists a $k$-uniform hypergraph $\mathcal H$ on $n=3k-3$ vertices such that
\[
\mathcal E\bigl(\mathcal S(\mathcal H)\bigr)
<
\mathcal E\bigl(\mathcal S(\mathcal C_n^k)\bigr).
\]
Hence, the hypergraph analogue of Haemers’ Conjecture fails for $k$-uniform hypergraphs.
\end{theorem}

\proof
Let $\mathcal H:=\mathcal{T}^k_3$, where $\mathcal{T}^k_3$ is the $k$-uniform hypertriangle of length $3$. Then $\mathcal H$ has $n=3k-3$ vertices. Therefore, it suffices to show that
\[
\mathcal E\bigl(\mathcal S(\mathcal{T}^k_3)\bigr)
<
\mathcal E\bigl(\mathcal S(\mathcal C_{3k-3}^k)\bigr).
\]
By Theorem \ref{HY2},
\begin{equation}\label{eq:complete-energy-new}
\mathcal E\bigl(\mathcal S(\mathcal C_{3k-3}^k)\bigr)
=
2(3k-4)\left(2\binom{3k-5}{k-2}-1\right).
\end{equation}
By Theorem \ref{HY3},
\begin{equation}\label{eq:triangle-energy-new}
\mathcal E\bigl(\mathcal S(\mathcal{T}^k_3)\bigr)
=
(3k-9)+4\sqrt{k^2-4}+\sqrt{k^2+6k-7}.
\end{equation}
We consider two cases.

\medskip
\noindent\textbf{Case 1:} $k=3,4,5$.

For $k=3$, by equation \eqref{eq:complete-energy-new},
\[
\mathcal E\bigl(\mathcal S(\mathcal C_{3k-3}^k)\bigr)
=
2\cdot5\cdot\left(2\binom{4}{1}-1\right)
=
70,
\]
and by equation \eqref{eq:triangle-energy-new},
\[
\mathcal E\bigl(\mathcal S(\mathcal{T}^k_3)\bigr)
=
0+4\sqrt{5}+\sqrt{20}
\approx 13.42.
\]
Hence
\[
\mathcal E\bigl(\mathcal S(\mathcal{T}^k_3)\bigr)
<
\mathcal E\bigl(\mathcal S(\mathcal C_{3k-3}^k)\bigr).
\]
For $k=4$, by equation \eqref{eq:complete-energy-new},
\[
\mathcal E\bigl(\mathcal S(\mathcal C_{3k-3}^k)\bigr)
=
2\cdot8\cdot\left(2\binom{7}{2}-1\right)
=
656,
\]
and by equation \eqref{eq:triangle-energy-new},
\[
\mathcal E\bigl(\mathcal S(\mathcal{T}^k_3)\bigr)
=
3+4\sqrt{12}+\sqrt{33}
\approx 22.60.
\]
Hence
\[
\mathcal E\bigl(\mathcal S(\mathcal{T}^k_3)\bigr)
<
\mathcal E\bigl(\mathcal S(\mathcal C_{3k-3}^k)\bigr).
\]
For $k=5$, by equation \eqref{eq:complete-energy-new},
\[
\mathcal E\bigl(\mathcal S(\mathcal C_{3k-3}^k)\bigr)
=
2\cdot11\cdot\left(2\binom{10}{3}-1\right)
=
5258,
\]
and by equation \eqref{eq:triangle-energy-new},
\[
\mathcal E\bigl(\mathcal S(\mathcal{T}^k_3)\bigr)
=
6+4\sqrt{21}+\sqrt{48}
\approx 31.28.
\]
Hence
\[
\mathcal E\bigl(\mathcal S(\mathcal{T}^k_3)\bigr)
<
\mathcal E\bigl(\mathcal S(\mathcal C_{3k-3}^k)\bigr).
\]

\medskip
\noindent\textbf{Case 2:} $k\ge 6$.

For integers $1\le b\le a$, we have
\[
\binom{a}{b}
=
\frac{a}{1}\cdot \frac{a-1}{2}\cdots \frac{a-b+1}{b}
=
\prod_{i=0}^{b-1}\frac{a-i}{i+1}
\]
For each $i=0,1,\dots,b-1$,
\[
a-i\ge a-b+1
\qquad\text{and}\qquad
i+1\le b.
\]
Therefore,
\[
\frac{a-i}{i+1}\ge \frac{a-b+1}{b}.
\]
Multiplying these inequalities for $i=0,1,\dots,b-1$, we obtain
\[
\binom{a}{b}
=
\prod_{i=0}^{b-1}\frac{a-i}{i+1}
\ge
\left(\frac{a-b+1}{b}\right)^b.
\]
Taking $a=3k-5~\text{and}~b=k-2$. Then $a-b+1=(3k-5)-(k-2)+1=2k-2$, and therefore
\[
\binom{3k-5}{k-2}
\ge
\left(\frac{2k-2}{k-2}\right)^{k-2}
=
\left(2+\frac{2}{k-2}\right)^{k-2}
\ge 2^{k-2}.
\]
Substituting this into equation \eqref{eq:complete-energy-new}, we obtain
\begin{equation}\label{eq:complete-lower-new}
\mathcal E\bigl(\mathcal S(\mathcal C_{3k-3}^k)\bigr)
\ge
2(3k-4)\bigl(2^{k-1}-1\bigr).
\end{equation}
Next, from equation \eqref{eq:triangle-energy-new},
\[
\sqrt{k^2-4}<k
\qquad\text{and}\qquad
\sqrt{k^2+6k-7}<k+3.
\]
Hence
\begin{equation}\label{eq:triangle-upper-new}
\mathcal E\bigl(\mathcal S(\mathcal{T}^k_3)\bigr)
<
(3k-9)+4k+(k+3)
=
8k-6.
\end{equation}
Therefore, it is enough to show that
\[
2(3k-4)\bigl(2^{k-1}-1\bigr)>8k-6,~~\text{for }k\ge 6.
\]
But this is immediate, since for $k\ge 6$,
\[
(2^{k-1}-1)\ge 31
\]
and hence
\[
2(3k-4)\bigl(2^{k-1}-1\bigr)
\ge
62(3k-4)
>
8k-6.
\]
Combining equations \eqref{eq:complete-lower-new} and \eqref{eq:triangle-upper-new}, we obtain
\[
\mathcal E\bigl(\mathcal S(\mathcal C_{3k-3}^k)\bigr)
>
\mathcal E\bigl(\mathcal S(\mathcal{T}^k_3)\bigr),~~\text{for all }k\ge 6.
\]
Therefore, in all cases $k\ge 3$, we have
\[
\mathcal E\bigl(\mathcal S(\mathcal{T}^k_3)\bigr)
<
\mathcal E\bigl(\mathcal S(\mathcal C_{3k-3}^k)\bigr).
\]
Hence there exists a $k$-uniform hypergraph $\mathcal H$ on $n=3k-3$ vertices such that
\[
\mathcal E\bigl(\mathcal S(\mathcal H)\bigr)
<
\mathcal E\bigl(\mathcal S(\mathcal C_n^k)\bigr).
\]
This completes the proof.\qe

\section{Random uniform hypergraphs and Seidel non-hypoenergeticity}

Random graphs are fundamental objects in probabilistic combinatorics. In the late 1950s and early 1960s, Erd\H{o}s and R\'enyi \cite{ER1959, ER1960, ER1961b} introduced the concept of a random graph, leading to several important models. Among the best known are the uniform model $G(n, M)$, in which a graph is chosen uniformly from all graphs with $n$ vertices and $M$ edges, and the binomial model $G(n,p)$, in which each of the $\binom{n}{2}$ possible edges is included independently with probability $p$. These models have been used extensively to study the asymptotic properties of graphs \cite{Bollobas2001, Spencer1990}. In 2007, Nikiforov \cite{VN2007} obtained the asymptotic behavior of the energy of a graph on $n$ vertices in the Erd\H{o}s-R\'enyi random graph model. Later, in 2011, Gutman \cite{IG2011} used Nikiforov's energy bound to characterize non-hypoenergetic graphs in the random graph model $G(n,p)$.

The Erd\H{o}s-R\'enyi model admits a natural analogue for hypergraphs. For integers $n\ge k\ge2$, the random $k$-uniform hypergraph $\mathcal{R}(n,k,p)$ is defined on the vertex set $[n]=\{1,2,\dots,n\}$ by including each $k$-subset independently as a hyperedge with probability $p$, see \cite{Behrisch2010, Cooley2018, Schmidt1975}.

Motivated by the classical notions of hypoenergetic and non-hypoenergetic graphs, in this section, we introduce the concepts of Seidel hypoenergetic and Seidel non-hypoenergetic hypergraphs. We then investigate Seidel non-hypoenergeticity for $k$-uniform hypergraphs in the random hypergraph model $\mathcal{R}(n, k, p)$.

\begin{definition}
A Bernoulli random variable with parameter $p\in(0,1)$ is a random variable $X$ taking values in $\{0,1\}$ such that
\[
\mathbb{P}(X=1)=p
\quad\text{and}\quad
\mathbb{P}(X=0)=1-p.
\]
\end{definition}
Random graphs are generated by assigning independent Bernoulli random variables to the possible edges. The classical Erd\H{o}s-R\'enyi model $G(n,p)$ treats each of the $\binom{n}{2}$ possible edges independently.

\begin{definition}
Let $n \in \mathbb{N}$ and $p \in [0,1]$. The Erd\H{o}s-R\'enyi random graph, denoted by $G(n,p)$, is a random graph on $n$ vertices in which each edge is included independently with probability~$p$. Formally, $G(n,p)$ is the probability space $\big( \mathcal{G}(n),\, 2^{\mathcal{G}(n)},\, \mathbb{P} \big)$, where $\mathcal{G}(n)$ denotes the collection of all graphs on $n$ labeled vertices. For $G \in \mathcal{G}(n)$ with edge set $E(G)$,
\begin{equation}
\label{eq:ERprob}
\mathbb{P}(\{G\}) 
  = \prod_{e \in E(G)} p 
    \prod_{e \notin E(G)} (1-p)
  = p^{|E(G)|}(1-p)^{\binom{n}{2}-|E(G)|}.
\end{equation}
\end{definition}
To illustrate the probabilistic construction of random graphs, let us examine the simplest nontrivial case of the Erd\H{o}s--R\'enyi model.

\begin{example}
Consider the Erd\H{o}s-R\'enyi random graph $G(3, \tfrac{1}{2})$. There are $\binom{3}{2} = 3$ possible edges among three labeled vertices $\{1,2,3\}$, namely  $e_1 = \{1,2\}$, $e_2 = \{1,3\}$, and $e_3 = \{2,3\}$.

Each edge is included independently with probability $p = \tfrac{1}{2}$. Thus, every graph $G$ on these three vertices has probability
\[
\mathbb{P}(\{G\}) = \Big(\tfrac{1}{2}\Big)^{|E(G)|} \Big(1-\tfrac{1}{2}\Big)^{3 - |E(G)|}= \Big(\tfrac{1}{2}\Big)^3 = \tfrac{1}{8}.
\]
Hence, all $2^3 = 8$ labeled graphs on $\{1,2,3\}$ are equally likely.
\end{example}

\begin{definition}
Let $G\sim G(n,p)$ be the Erd\H{o}s-R\'enyi random graph, i.e., $G$ is sampled from the Erdős–Rényi random graph model $G(n, p)$. For each unordered pair $\{i,j\}$ of distinct vertices, define
the random variable
\[
a_{ij} =
\begin{cases}
1, & \text{with probability } p,\\
0, & \text{with probability } 1-p,
\end{cases}
\]
and assume that the family $\{a_{ij} : 1\le i<j\le n\}$ is independent. Thus, the adjacency matrix of $G$ is $A=(a_{ij})$ with $a_{ji}=a_{ij}$ and $a_{ii}=0$.
\end{definition}
In 2007, Nikiforov obtained the following result for the Erd\H{o}s-R\'enyi random graph model, describing the asymptotic behavior of the energy of a graph on $n$ vertices.

\begin{lemma}\rm\big(Nikiforov, \cite[Page 1474, Line 5]{VN2007}\big)
Let $G \sim G\!\left(n,\tfrac12\right)$, and let $A=A(G)$ be its adjacency matrix. Then for almost all graphs $G$ with $n$ vertices, the energy of a graph $G$ is
\[
\mathcal{E}(G)
= \left(\frac{4}{3\pi}+o(1)\right)n^{3/2}.
\]
\end{lemma}

\begin{remark}
The phrase ``for almost all graphs $G$'' means that
\[
\mathbb{P}\!\left(
\mathcal{E}(G)
= \left(\frac{4}{3\pi}+o(1)\right)n^{3/2}
\right) \longrightarrow 1
\quad \text{as } n\to\infty,
\]
where $G \sim G\!\left(n,\tfrac12\right)$.
\end{remark}
In 2011, Gutman introduced the notions of hypoenergetic and non-hypoenergetic graphs \cite{IG2011}.

\begin{definition}\label{dr}
A graph $G$ on $n$ vertices is called hypoenergetic if $\mathcal{E}(G)<n$ and is called non-hypoenergetic if $\mathcal{E}(G)\ge n$, where $\mathcal{E}(G)$ denotes the classical graph energy.
\end{definition}
In 2011, Gutman proved the following theorem.

\begin{theorem}\rm\big(Gutman, \cite[Theorem~4.4]{IG2011}\big)\label{Gutrman}
Almost all graphs are non-hypoenergetic.    
\end{theorem}
Therefore, Theorem \ref{Gutrman} provides a characterization of all non-hypoenergetic graphs. Motivated by the above discussion, we first extend the concepts of hypoenergeticity and non-hypoenergeticity to hypergraphs using the Seidel matrix of hypergraphs, following Definition \ref{dr}.

\begin{definition}
A hypergraph $\mathcal{H}$ on $n$ vertices is said to be Seidel hypoenergetic if $\mathcal{E}(\mathcal{S(H)})<n$ and Seidel non-hypoenergetic if $\mathcal{E}(\mathcal{S(H)})\ge n$.
\end{definition}
We now pose the following question.

\begin{question}\label{trtrtr}
Are almost all $k$-uniform hypergraphs Seidel non-hypoenergetic?
\end{question}
\textbf{Answer.} The answer is affirmative. 

Before we answer Question \ref{trtrtr}, we first need to study some basic concepts of random hypergraphs.

\begin{definition}
The $k$-uniform random hypergraph $\mathcal{R}(n,k,p)$ generalizes $G(n,p)$. Given integers $n \ge 2$ and $1 \le k \le n$, $\mathcal{R}(n,k,p)$ is the probability space $\big( \mathcal{H}^k_n,\, 2^{\mathcal{H}^k_n},\, \mathbb{P} \big)$, where $\mathcal{H}^k_n$ denotes the set of all $k$-uniform hypergraphs with vertex set $[n]$. For $\mathcal G \in \mathcal{H}^k_n$ with edge set $E(\cal G)$,
\begin{equation}
\label{eq:hyperprob}
\mathbb{P}(\{\mathcal G\})=\prod_{e \in E(\mathcal G)} p \prod_{e \in \binom{[n]}{k}\setminus E(\mathcal G)}(1-p)= p^{|E(\mathcal G)|}(1-p)^{\binom{n}{k}-|E(\mathcal G)|}.
\end{equation}
Clearly, when $k=2$, $\mathcal{R}(n,2,p)$ coincides with the Erd\H{o}s-R\'enyi graph~$G(n,p)$.
\end{definition}

To illustrate the probabilistic model for random hypergraphs, let us consider the simplest nontrivial case of a $3$-uniform random hypergraph.

\begin{example}
Consider the $3$-uniform random hypergraph $\mathcal{R}(4,3,\tfrac{1}{2})$. The vertex set is $[4]=\{1,2,3,4\}$. All possible $3$-element subsets of $[4]$ are
\[
\binom{[4]}{3}=\big\{\{1,2,3\},\,\{1,2,4\},\,\{1,3,4\},\,\{2,3,4\}\big\},
\]

so there are $\binom{4}{3}=4$ possible hyperedges in total.

Each potential hyperedge is included independently with probability $p=\tfrac{1}{2}$. Hence, the probability of any particular $3$-uniform hypergraph $\mathcal{H}$ with $|E(\mathcal{H})|$ edges is
\[
\mathbb{P}(\{\mathcal{H}\}) 
  = \Big(\tfrac{1}{2}\Big)^{|E(\mathcal{H})|}
    \Big(1-\tfrac{1}{2}\Big)^{4-|E(\mathcal{H})|}
  = \Big(\tfrac{1}{2}\Big)^4
  = \tfrac{1}{16}.
\]
Therefore, all $2^4=16$ possible $3$-uniform hypergraphs on $[4]$ are equally likely.  
\end{example}

\begin{definition}
A random variable $X$ is said to have the binomial distribution with parameters $M$ and $p$, denoted by $\mathrm{Binomial}(M,p)$, if
\[
\mathbb{P}(X = k) = \binom{M}{k} p^k (1-p)^{M-k},
\qquad k = 0,1,\dots,M.
\]
\end{definition}

\begin{definition}
Let $X$ be a discrete random variable. The mean (or expected value) of $X$ is defined by
\[
\mathbb{E}[X] = \sum_x x\,\mathbb{P}(X=x),
\]
provided the sum converges. The variance of $X$ is defined by
\[
\operatorname{Var}(X)=\mathbb{E}\big[(X-\mathbb{E}[X])^2\big]
=\mathbb{E}[X^2]-(\mathbb{E}[X])^2.
\]
\end{definition}

\begin{remark}\label{reeee}

The following holds:
\begin{enumerate}
\item If $X \sim \mathrm{Binomial}(M,p)$, then
\[
\mathbb{E}[X]=Mp
\quad \text{and} \quad
\operatorname{Var}(X)=Mp(1-p).
\]

\item Let $\mathcal{H}=(V, E)\sim \mathcal{R}(n,k,p)$ be a random $k$-uniform hypergraph on $n$ vertices, where each $k$-subset of $V$ is included independently as a hyperedge with probability $p$. For a fixed distinct $i,j\in V$, the random variable $c_{ij}$ counts the number of hyperedges containing $\{i,j\}$. There are exactly $\binom{n-2}{k-2}$ such $k$-subsets of $V$, and each is included independently with probability $p$. Hence
\[
c_{ij}\sim \mathrm{Binomial}\!\left(\binom{n-2}{k-2},\,p\right),~~ 1 \le i < j \le n,~~\mathbb{E}[c_{ij}]=p\binom{n-2}{k-2}.
\]
\end{enumerate}
\end{remark}
The following relations between functions are classic.

\begin{definition}
Let $f,g:\mathbb{N}\to\mathbb{R}$ be two functions with $g(n)\neq 0$ for all sufficiently large $n$. Then, as $n\to\infty$:

\begin{enumerate}
    \item $f(n)=O(g(n))$ if there exist constants $C>0$ and $n_0\in\mathbb{N}$ such that
    \[
    |f(n)|\le C|g(n)| \qquad \text{for all } n\ge n_0.
    \]

    \item $f(n)=o(g(n))$ if
    \[
    \lim_{n\to\infty}\frac{f(n)}{g(n)}=0.
    \]
\end{enumerate}
\end{definition}

\begin{lemma}\label{newla}
Let $n>2$ be a positive integer. Then for a fixed positive integer $k\geq3$, 
\begin{center}
$\binom{n-2}{k-2}=\frac{n^{k-2}}{(k-2)!}\prod\limits_{r=2}^{k-1}\big(1-\frac{r}{n}\big)$
\end{center}    
\end{lemma}

\proof For a positive integer $n>2$ and a fixed positive integer $k\geq3$,
\begin{align*}
\binom{n-2}{k-2}
&=\frac{(n-2)!}{(n-k)!(k-2)!},\\
&=\frac{(n-2)(n-3)(n-4)\cdots\cdots(n-k+2)(n-k+1)}{(k-2)!},\\
&=\frac{n^{k-2}\big(1-\frac{2}{n}\big)\big(1-\frac{3}{n}\big)\big(1-\frac{4}{n}\big)\cdots\cdots\big(1-\frac{k-2}{n}\big)\big(1-\frac{k-1}{n}\big)}{(k-2)!},\\
&=\frac{n^{k-2}}{(k-2)!}\prod\limits_{r=2}^{k-1}\big(1-\frac{r}{n}\big).
\end{align*}
This completes the proof.\qe

\begin{lemma}\label{ll}
Let $n>2$ be a large positive integer. Then for a fixed positive integer $k\geq3$, 
\begin{center}
$\binom{n-2}{k-2}\approx\frac{n^{k-2}}{(k-2)!}$
\end{center}
\end{lemma}
\proof Let us consider the logarithm of the expression $\prod\limits_{r=2}^{k-1}\big(1-\frac{r}{n}\big)$ appearing in the right-hand side of the formula in Lemma \ref{newla},
\begin{align*}
\log\Big(\prod\limits_{r=2}^{k-1}\big(1-\frac{r}{n}\big)\Big)
&=\sum\limits_{r=2}^{k-1}\log\big(1-\frac{r}{n}\big),\\
&=-\sum\limits_{r=2}^{k-1}\Big(\frac{r}{n}+O(\frac{1}{n^2})\Big),\\
&=-\frac{\sum\limits_{r=2}^{k-1}r}{n}+O(\frac{1}{n^2}),\\
&=-\frac{k^2-k-2}{2n}+O(\frac{1}{n^2}).
\end{align*}
Applying the exponentiating, 
\begin{equation}\label{b2}
\prod\limits_{r=2}^{k-1}\big(1-\frac{r}{n}\big)=\exp\Big(-\frac{k^2-k-2}{2n}+O(\frac{1}{n^2})\Big).
\end{equation}
Using the first-order Taylor expansion for small $x$,
\begin{align*}
\exp(x)
&=1+\frac{x}{1!}+\frac{x^2}{1!}+\frac{x^3}{1!}+\frac{x^4}{1!}+\cdots.
\end{align*}
Substitute $x=-\frac{k^2-k-2}{2n}+O(\frac{1}{n^2})$ into the Taylor expansion of $\exp(x)$,
\small{
\begin{align*}
\exp\Big(-\frac{k^2-k-2}{2n}+O(\frac{1}{n^2})\Big)
&=1+\frac{-\frac{k^2-k-2}{2n}+O(\frac{1}{n^2})}{1!}+\frac{\Big(-\frac{k^2-k-2}{2n}+O(\frac{1}{n^2})\Big)^2}{2!}+\cdots,\\
&=1+\Big((-\frac{k^2-k-2}{2n})+O(\frac{1}{n^2})\Big)+\frac{1}{2}\Big(-\frac{k^2-k-2}{2n}+O(\frac{1}{n^2})\Big)^2+\cdots,\\
\end{align*}}
Since the term $\frac{1}{p!}\Big(-\frac{k^2-k-2}{2n}+O(\frac{1}{n^2})\Big)^p=O(\frac{1}{n^2})$, for all $p=2,3,\ldots$
\begin{align*}
=1-\frac{k^2-k-2}{2n}+O(\frac{1}{n^2}).
\end{align*}
Therefore,
\begin{equation}\label{Bbb3}
\exp\Big(-\frac{k^2-k-2}{2n}+O(\frac{1}{n^2})\Big)=1+O(\frac{1}{n}).
\end{equation} 
Using equations \eqref{b2}, \eqref{Bbb3}, and Lemma \ref{newla} we have
\begin{center}
$\binom{n-2}{k-2}=\frac{n^{k-2}}{(k-2)!}\big(1+O(\frac{1}{n})\big).$
\end{center}
Since $O(\frac{1}{n})$ is goes to zero as $n\rightarrow\infty$ i.e., for large $n$, which implies that asymptotically $\binom{n-2}{k-2}$ behaves like 
\begin{center}
$\binom{n-2}{k-2}\approx\frac{n^{k-2}}{(k-2)!}$
\end{center}
This completes the proof.\qe

Chernoff’s inequality is a fundamental result in probability theory, which dates back to \cite{HC1952}, 1952. The following lemma is known as Chernoff's inequality.

\begin{lemma}\rm\cite[Chernoff’s inequality]{HC1952}\label{C1}
For a binomial random variable $X$ with mean $\mu$ and for any $\epsilon>0$:
\[
\mathbb{P}\Big(|X-\mu|>\epsilon\mu\Big)\leq2~\exp(-\frac{\epsilon^2\mu}{3}).
\]
\end{lemma}

The following lemma guarantees that every pair of $(i,j)\in V$, $c_{ij}$ is very close to its expected value $p\binom{n-2}{k-2}$ with overwhelming probability.

\begin{lemma}\label{AA1}
Let $\mathcal{H}=(V,E)\sim\mathcal{R}(n,k,p)$ be a random $k$-uniform hypergraph on $n$ vertices with
fixed integers $k\ge 3$ and $p\in(0,1)$. If $i\neq j$ then
\begin{center}
$c_{ij}=p\Big(1+o(1)\Big)\binom{n-2}{k-2}$.   
\end{center}
for almost all $k$-uniform hypergraphs.
\end{lemma}

\proof From Remark \ref{reeee} (ii),
\[
c_{ij}\sim\mathrm{Binomial}\Bigl(\binom{n-2}{k-2},\,p\Bigr),
\quad
\mu:=\mathbb{E}[c_{ij}]=p\binom{n-2}{k-2}.
\]
Since $\epsilon>0$ is arbitrary, let us take $\epsilon=\sqrt{\left(\frac{9\log n}{p\binom{n-2}{k-2}}\right)}$. Applying Lemma \ref{C1},
\[
\mathbb{P}\big(|c_{ij}-\mu|>\epsilon\mu\big)
\le 2\exp\left(-\frac{\epsilon^2\mu}{3}\right)
= 2\exp\bigl(-3\log n\bigr)
= \frac{2}{n^3}.
\]
Let us define the event $\mathcal{U}_{ij}$ for each pair of vertices $\{i,j\}$ to be
\[
\mathcal{U}_{ij}:=\big|c_{ij}-\mu\big|>\epsilon\mu.
\]
There are $\binom{n}{2}$ unordered pairs $\{i,j\}$. By the union bound of probability, we have
\begin{align*}
\mathbb{P}\Big(\bigcup_{i<j}\mathcal{U}_{ij}\Big)
&\le \sum_{i<j}\mathbb{P}(\mathcal{U}_{ij}),\\
&\le \binom{n}{2}\cdot\frac{2}{n^3},\to0~~\text{as}~n\to\infty\\
&= o(1).
\end{align*}
Thus, with probability $1-o(1)$, the complement of this union occurs, i.e.,
\[
\big|c_{ij}-\mu\big|\le \epsilon\mu
\quad\text{for all }1\le i<j\le n.
\]
Equivalently with probability tending to $1$ as $n\to\infty$, we have
\[
c_{ij} = \mu(1\pm\epsilon)
= p\binom{n-2}{k-2}(1\pm\epsilon)
\quad\text{for all }i\ne j.
\]
Now we need to show that $\epsilon=o(1)$. By Lemma \ref{ll},
\[
\binom{n-2}{k-2}\approx\frac{n^{k-2}}{(k-2)!},
\]
so
\[
\epsilon^2
=\frac{9\log n}{p\binom{n-2}{k-2}}
\approx \frac{9(k-2)!}{p}\frac{\log n}{n^{k-2}}\to0~~\text{as}~n\to\infty,
\]
and hence
\[
\epsilon
= o(1)~~\text{for}~k\ge 3.
\]
Therefore
\[
c_{ij} = p\binom{n-2}{k-2}\bigl(1\pm \epsilon\bigr)=p\binom{n-2}{k-2}\bigl(1+o(1)\bigr),
\]
uniformly over all pairs $\{i,j\}$, since the right-hand side does not depend on $i,j$ as claimed.\qe

The following theorem answers \textbf{Question \ref{trtrtr}}.
\begin{theorem}\label{new}
Let $k\ge3$. Almost all $k$-uniform hypergraphs are Seidel non-hypoenergetic.    
\end{theorem}
\proof The Seidel matrix of a hypergraph $\mathcal{H}=(V,~E)$ is $\mathcal{S}(\mathcal{H})=J_n-I_n-2\mathcal{A}(\mathcal{H})$. Let $i,j\in V$. The $ij$-th entry of the Seidel matrix $\mathcal{S}(\mathcal{H})=[\mathcal{S}(\mathcal{H})_{ij}]_{n\times n}$ is 
\begin{align*}
\mathcal{S}(\mathcal{H})_{ij}=1-2c_{ij},~\text{for}~i\neq j  \end{align*}
By Lemma \ref{AA1}, uniformly for all $i\neq j$,
\begin{align*}
\mathcal{S}(\mathcal{H})_{ij}
&=1-2p\binom{n-2}{k-2}+o\bigg(p\binom{n-2}{k-2}\bigg).
\end{align*}
Then
\begin{align*}
\|\mathcal S(\mathcal H)\|_F^2
&=\sum_{i\neq j}\mathcal S(\mathcal H)_{ij}^2\\
&=\sum_{i\neq j}\Bigg[1-2p\binom{n-2}{k-2}+o\bigg(p\binom{n-2}{k-2}\bigg)\Bigg]^2\\
&=n(n-1)\Bigg[1-2p\binom{n-2}{k-2}+o\bigg(\Big(2p\binom{n-2}{k-2}-1\Big)\bigg)\Bigg]^2\\
&=n(n-1)\Bigg[1-2p\binom{n-2}{k-2}+\bigg(2p\binom{n-2}{k-2}-1\bigg)o(1)\Bigg]^2\\
&=n(n-1)\bigg(2p\binom{n-2}{k-2}-1\bigg)^2\big(1+o(1)\big).
\end{align*}
Therefore,
\[
\|\mathcal S(\mathcal H)\|_F
=\bigg|2p\binom{n-2}{k-2}-1\bigg|\sqrt{n(n-1)\big(1+o(1)\big)}.
\]
By Lemma \ref{lem:sand}, $\mathcal E(\mathcal S(\mathcal H))\ge \|\mathcal S(\mathcal H)\|_F$, implies
\begin{equation}\label{is}
\mathcal E(\mathcal S(\mathcal H))
\ge\bigg|2p\binom{n-2}{k-2}-1\bigg|\sqrt{n(n-1)\big(1+o(1)\big)}.    
\end{equation}
Moreover, $p\in(0,1)$ is fixed and $k\ge 3$, we have $p\binom{n-2}{k-2}\to\infty$ as $n\to\infty$,
\begin{align*}
\bigg|2p\binom{n-2}{k-2}-1\bigg|
&=\bigg|2p\binom{n-2}{k-2}\Big(1-\frac{1}{2p\binom{n-2}{k-2}}\Big)\bigg|,\\
&=2p\binom{n-2}{k-2}\big(1+o(1)\big).
\end{align*}
By equation \eqref{is}, 
\[
\mathcal E(\mathcal S(\mathcal H))
\ge 2p\binom{n-2}{k-2}\sqrt{n(n-1)}\,\big(1+o(1)\big)^{\frac{3}{2}}.
\]
By Lemma \ref{ll}, 
\[
\binom{n-2}{k-2}\approx\frac{n^{k-2}}{(k-2)!},
\]
we obtain
\[
2p\binom{n-2}{k-2}\sqrt{n(n-1)}\approx\frac{2p}{(k-2)!}n^{k-2}\cdot n=\frac{2p}{(k-2)!}n^{k-1}.
\]
Since $k\ge 3$, we have $k-1\ge 2$, and therefore for sufficiently large $n$,
\[
\mathcal{E}(\mathcal{S}(\mathcal{H}))> n.
\]
This completes the proof.\qe

\section{Concluding remark}
A graph $G$ on $n$ vertices is called hyperenergetic if $\mathcal{E}(G)>\mathcal{E}(K_n) = 2n-2$ \cite[Definition~7.1]{IG2001}. In 2001, Gutman showed that there exist many graphs whose energy is more than the energy of the complete graphs, i.e., $n$-vertex graphs for which $\mathcal{E}(G)>\mathcal{E}(K_n)$ \cite[Proposition 7.2, Theorem 7.6, Corollary 7.8]{IG2001}. Though the characterization of the $n$-vertex graph(s) with maximum value of $\mathcal{E}$ is an open problem \cite[Abstract]{IG2001}. In the recent proof of \emph{Haemers’ Conjecture} for graph \cite{S2020, E2024}, the notion of hyperenergetic graphs does not play a significant role, since with respect to the Seidel matrix of a graph, every graph automatically satisfies the hyperenergetic property. However, Theorem \ref{thm:haemers-fails} shows that \emph{Haemers’ Conjecture} fails for every $k$-uniform hypergraph $\mathcal{H}$ when the Seidel matrix $\mathcal{S}(\mathcal{H})$ is used. This naturally motivates extending the concept of hyperenergeticity to the hypergraph setting via Seidel energy.

Accordingly, a hypergraph $\mathcal{H}$ on $n$ vertices is called \emph{Seidel hyperenergetic} if $\mathcal{E}\!\left(\mathcal{S}(\mathcal{H})\right)>\mathcal{E}\!(\mathcal{S}(\mathcal{C}^k_n))$ and \emph{Seidel non-hyperenergetic} if $\mathcal{E}\!\left(\mathcal{S}(\mathcal{H})\right)<\mathcal{E}\!(\mathcal{S}(\mathcal{C}^k_n))$.

We conclude with several open questions.
\begin{question}
Characterize all $k$-uniform hypergraphs that are Seidel hyperenergetic and those that are Seidel non-hyperenergetic.
\end{question}

\begin{question}
Characterize the $k$-uniform hypergraph that maximizes (minimizes) the Seidel energy among all $k$-uniform hypergraphs with $n$ vertices.
\end{question}

\begin{question}
Characterize the $r$-regular $k$-uniform hypergraph that maximizes (minimizes) the Seidel energy among all $r$-regular $k$-uniform hypergraphs with $n$ vertices.
\end{question}

\begin{question}
Characterize the $r$-regular $k$-uniform hypergraph that maximizes (minimizes) the Frobenius norm of the Seidel matrix among all $r$-regular $k$-uniform hypergraphs with $n$ vertices.
\end{question}

\section*{Declaration of competing interest}
The authors declare that they have no competing interests.

\section*{Acknowledgements}
The authors express their sincere gratitude to the Department of Mathematics at Bar-Ilan University. The second author also gratefully acknowledges the postdoctoral financial support provided by Bar-Ilan University, Israel.

\section*{Data availability}
No data were used in the research described in this article.

\end{document}